\UseRawInputEncoding
\documentclass[12pt,draft,reqno]{amsart}

\usepackage[all]{xy}
\usepackage{amsthm,array,amssymb,amscd,amsfonts,latexsym}

\usepackage{enumerate}

\usepackage{paralist}
\usepackage{mathrsfs}
\usepackage{mathrsfs}
\usepackage{dsfont}
\usepackage{bbold}
\usepackage{mathbbol}
\usepackage[all]{xy}
\usepackage{enumerate}

{\catcode`\@=11
\gdef\n@te#1#2{\leavevmode\vadjust{%
 {\setbox\z@\hbox to\z@{\strut#1}%
  \setbox\z@\hbox{\raise\dp\strutbox\box\z@}\ht\z@=\z@\dp\z@=\z@%
  #2\box\z@}}}
\gdef\leftnote#1{\n@te{\hss#1\quad}{}}
\gdef\rightnote#1{\n@te{\quad\kern-\leftskip#1\hss}{\moveright\hsize}}
\gdef\?{\FN@\qumark}
\gdef\qumark{\ifx\next"\DN@"##1"{\leftnote{\rm##1}}\else
 \DN@{\leftnote{\rm??}}\fi{\rm??}\next@}}

\DeclareFontFamily{OT1}{wncyr}{\hyphenchar\font45
}
\DeclareFontShape{OT1}{wncyr}{m}{n}{%
   <5> <6> <7> <8> <9> gen * wncyr
   <10> <10.95> <12> <14.4> <17.28> <20.74>  <24.88>wncyr10}{}
\DeclareFontShape{OT1}{wncyr}{m}{it}{%
   <5> <6> <7> <8> <9> gen * wncyi
   <10> <10.95> <12> <14.4> <17.28> <20.74> <24.88> wncyi10}{}
\DeclareFontShape{OT1}{wncyr}{m}{sc}{%
   <5> <6> <7> <8> <9> <10> <10.95> <12> <14.4>
   <17.28> <20.74> <24.88>wncysc10}{}
\DeclareFontShape{OT1}{wncyr}{b}{n}{%
   <5> <6> <7> <8> <9> gen * wncyb
   <10> <10.95> <12> <14.4> <17.28> <20.74> <24.88>wncyb10}{}
\input cyracc.def

\DeclareMathSizes{9}{9}{7}{5}

\theoremstyle{plain}

\newtheorem{theorem}{Theorem}[section]

\newtheorem{lemma}{Lemma}[section]

\theoremstyle{definition}

\newtheorem{nothing*}[theorem]{}
\newtheorem{subnothing*}[sub]{}

\newtheorem*{quesnonumber}{Question}

\newtheorem*{remarknonumber}{Remark}

\theoremstyle{remark}

\newcommand{\cc}{\raise .4pt \hbox{{$\scriptstyle{\bullet}$}}}

\begin{document}

\title[Group embeddings]{Embeddings of automorphism groups\\ of free groups
into automorphism groups\\ of affine
algebraic varieties}
\author[Vladimir\;L.\;Popov]{Vladimir\;L.\;Popov}
\address{Steklov Mathematical Institute, Russian Academy of Sciences, Gub\-kina 8, Moscow 119991, Russia}
\email{popovvl@mi-ras.ru}

\dedicatory{To the memory of A.\;N.\;Parshin}

\maketitle

{\def\thefootnote{\relax}

\begin{abstract}
A new infinite series of rational affine algebraic va\-rieties is constructed whose automorphism group contains the auto\-morphism group ${\rm Aut}(F_n)$ of the free group $F_n$ of rank $n$. The automorphism groups of such varieties are nonlinear and contain the braid group $B_n$ on $n$ strands for $n\geqslant 3$, and are nonamenable for $n\geqslant 2$. As an application, it is proved that for $n\geqslant 3$, every Cremona group of rank $\geqslant 3n-1$ contains the groups ${\rm Aut}(F_n)$ and $B_n$. This bound is 1 better than the one published earlier by the author; with respect to $B_n$
the order of its growth rate is one less than that of the bound following from the paper by D. Krammer.
The basis of the construction are triplets
$(G, R, n)$, where $G$ is a connected semisimple algebraic group and $R$ is a closed subgroup of its maximal torus.
\end{abstract}

\begin{section}{\bf Introduction}

 The trend of the last decade has been the study of abstract-algebraic, topological, algebro-geometric, and dynamical properties of automor\-phism groups of algebraic varieties. This paper is related to this topic and
continues the research started in
  author's paper \cite{22}.

In \cite{22}, an infinite series of irreducible algebraic varieties is const\-ructed in whose automorphism group embeds the automorphism group of a free group $F_n$ of rank $n$. This has applications to the problems of linearity and amenability of automorphism groups of algebraic varieties and that of the embeddability of various groups into Cremona groups.
To formulate the results obtained in this paper, we recall the construc\-tion introduced in \cite{22}.

Consider a connected algebraic group $G$. Denote
\begin{equation*}\label{G}
\begin{gathered}
\mbox{ $X$ is the group variety of the algebraic group}\\[-1.3mm]
G^{n}:=G\times\cdots\times G\quad \mbox{($n$ times).}
\end{gathered}
\end{equation*}
We fix in $F_n$ a free system of generators $f_1,\ldots, f_n$.
For any $w
\in F_n$ and
\begin{equation}\label{gh}
x=(g_1, \ldots, g_n)\in X,\quad g_j\in G \;\mbox{for all $j$},
\end{equation}
  denote by $w(x)$ the element of $G$
obtained from the word $w$ in $f_1,\ldots, f_n$ by replacing $f_j$ with $g_j$ for each\;$j$.
For each $\sigma\in {\rm Aut}(F_n)$, the mapping
\begin{equation}\label{mor}
\sigma_X\colon X\to X,\quad x\mapsto (\sigma(f_1)(x),\ldots, \sigma(f_n)(x)).
\end{equation}
is an automorphism of the algebraic variety $X$ (but not, in general, of the group $G^n$). The mapping $\sigma\mapsto (\sigma^{-1})_X$ is a group homomorphism ${\rm Aut}(F_n)\to {\rm Aut}(X)$. It defines an action of the group ${\rm Aut}(F_n)$ by automorphisms of the variety $X$ commuting
 with the diagonal action
$G$ on $X$ by conjugation. Let us assume that for the restriction of this action to a
 closed subgroup $R$ of the group $G$, there is a categorical quotient
\begin{equation}\label{Luna}
\pi_{{X}/\!\!/R}^{\ }\colon {X}\to {X}/\!\!/R
\end{equation}
(for example, this property holds if $R$ is finite, see \cite[Prop. 19; p.\,50, Expl.\,2)]{22-21}, or if $G$ is affine and $R$ is reductive,
see \cite[4.4]{PV}).
Then it follows from the definition of categorical quotient (see \cite[Def.\,4.5]{PV}) that $\sigma_X$ descends to a uniquely defined automorphism $\sigma_{X/\!\!/R}$ of $X /\!\!/R$, which has the property
\begin{equation}\label{pssp}
\pi_{{X}/\!\!/R}^{\ }\circ\sigma_X=\sigma_{X/\!\!/R}^{\ }\circ\pi_{{X}/\!\!/R}^{\ }.
\end{equation}
In this case, there arises a group homomorphism
\begin{equation}\label{action3}
{\rm Aut}(F_n)\to {\rm Aut}({X/\!\!/R}), \quad \sigma\mapsto (\sigma^{-1})_{X/\!\!/R},
\end{equation}
defining the action of the group ${\rm Aut}(F_n)$ by automorphisms of the variety ${X/\!\!/R}$. For some (but not all) $G$ and $R$,
homomorphism \eqref{action3} is an embedding.
Namely, in \cite{22} it is proved that
\begin{enumerate}[\hskip 3.2mm \rm(a)]
\item in the following cases, homomorphism \eqref{action3} is an embedding:
\begin{enumerate}[\hskip 0mm $\cc$]
\item $G$ is nonsolvable and $R$ is finite;
\item $G$ is reductive, $R=G$, $n=1$ and $G$ contains a connected simple normal subgroup of one of the following types:
    \begin{equation}\label{tyty}
    \mbox{${\sf A}_{\ell}$ with $\ell\geqslant 2$,\;
${\sf D}_\ell$ with odd $\ell$,\;
${\sf E}_6$;}
\end{equation}
\end{enumerate}
\item in the following cases, homomorphism \eqref{action3} is not an embedding:
\begin{enumerate}[\hskip 0mm $\cc$]
\item $G$ is solvable, $R$ is finite and $n\geqslant 3$,
\item $G$ is reductive, $R=G$ and either $n\geqslant 2$, or $n=1$ and $G$ does not contain a connected simple normal subgroup of either of types\;{\rm\eqref{tyty}}.
\end{enumerate}
\end{enumerate}

This leads to the following general question:
\begin{quesnonumber}
Is it possible to classify the triples $(G, R, n)$, where $G$\ is a connected reductive algebraic group, $R$ is its closed subgroup, and $n$ is a positive integer, for which homomorphism {\rm \eqref{action3}} is an embedding?
\end{quesnonumber}

The main result of the present paper is Theorem \ref{mate}, in which the next step after \cite{22} is taken towards answering the question posed: we add one more class to the triples of the specified type found in \cite{22}:

\begin{theorem}\label{mate}
Let $G$ be a connected semisimple algebraic group and let $R$ be a closed subgroup of its maximal torus.
Then homomorphism {\rm \eqref{action3}} is an embedding.
\end{theorem}

As an application
we obtain the following Theorem \ref{mateappl}, in which $B_n$ denotes the braid group on $n$ strands.

\begin{theorem}\label{mateappl}
We keep the assumptions of Theorem {\rm \ref{mate}}. Then $X/\!\!/R$
is an affine algebraic variety such that
\begin{enumerate}[\hskip 4.5mm\rm(a)]
\item ${\rm Aut}(X/\!\!/R)$ contains ${\rm Aut}(F_n)$,
\item ${\rm Aut}(X/\!\!/R)$ is nonlinear for $n\geqslant 3$,
\item ${\rm Aut}(X/\!\!/R)$
contains the braid group $B_n$ for $n\!\geqslant\!3$.
\item ${\rm Aut}(X/\!\!/R)$
is nonamenable for $n\!\geqslant\!2$.
\end{enumerate}
\end{theorem}

We also explore the rationality problem:

\begin{theorem}\label{rationa}
The variety $X/\!\!/R$ from Theorems {\rm \ref{mate}, \ref{mateappl}} is rational.
\end{theorem}

As an application
we strengthen by $1$ the bounds obtained in \cite[Cor.\,9 and Rem.\,10]{22}:

\begin{theorem}\label{corcor}
For any integer $n\geqslant 1$,
the Cremona group of rank $\geqslant 3n-1$ contains the group ${\rm Aut}(F_n)$ and, for $n\geqslant 3$, the braid group\;$B_n$.
\end{theorem}

\begin{remarknonumber}
As is proved in \cite{14} by D. Krammer, the braid group $B_n$ embeds into ${\rm GL}_{n(n-1)/2}$. Hence $B_n$ embeds into the Cremona group of rank $n(n-1)/2$. The order of the growth rate of this bound for the minimal rank of the Cremona group containing $B_n$ is one bigger  than that of the bound from Theorem \ref{corcor}.
\end{remarknonumber}

The proofs of Theorems \ref{mate}--\ref{corcor} are given in Section \ref{proofs}.

\end{section}

\begin{section}{\bf Conventions and notation}

 In what follows, algebraic varieties are considered over an algebrai\-cal\-ly closed field $k$.
We use the results of paper \cite{Lu2} and statement \cite[Prop. 3.4]{PV} obtained under the condition ${\rm char}(k)=0$. Therefore, we also assume that this condition holds. With respect to algebraic geometry and algebraic groups we follow \cite{3}.

The identity element of a group considered in multiplicative nota\-tion is denoted by $e$ (it will be clear from the context which group is meant).

The statement that a group $A$ contains a group $B$
means the exis\-tence of a group monomorphism $
B\hookrightarrow A$, by which $B$ is identified with its image.

${\mathscr C}(A)$ is the center of a group $A$.

$A\cdot m$ and $A_m$ are respectively, the orbit and the stabilizer of a point $m$ with respect to a considered action of a group $A$ (it will be clear from the context which action is meant).

The {\it kernel of an action} $\alpha\colon A\times M\to M$ of a group $A$ on a set $M$
is the following normal subgroup of $A$:
$$\ker(\alpha):=\{a\in A\mid a\cdot m=m\;\mbox{for all $m\in M$}\}.$$

By homomorphisms of algebraic groups we mean algebraic homomor\-phisms, and by their actions on algebraic varieties we mean algebraic actions. In particular, for an algebraic group $A$, we denote by ${\rm Aut}(A)$ the group of its algebraic automorphisms.

The multiplicative group $k^\times$ of the field $k$ is considered as
the algeb\-raic group $\mathbb G_m$, and its additive group is considered as $\mathbb G_a$.

\end{section}

\begin{section}{\bf Terminology and some general results}\label{term}

Recall (see \cite{PV}) the terminology and results used below, which con\-cern an action of $\alpha$ of an algebraic group $H$ on an irreducible algebraic variety $Y$.

(a) An action $\alpha$ is said to be {\it stable} if there is a nonempty open set in $Y$, the $H$-orbits of whose points are closed in\;$Y$.

(b) A subgroup $H_*$ of $H$ is called the {\it stabilizer in general posisiton} (s.g.p.) of the action $\alpha$ if there is a nonempty open set in $Y$ such that for any its point $y$, the subgroups $H_y$ and $H_*$ are conjugate in $H$.

(c) If $H$ is reductive and $Y$ is smooth and affine, then

\hskip 6.5mm $\cc$  The s.g.p. exists.

\hskip 6.5mm $\cc$   The varieties $Y$ and $Y/\!\!/H$ are endowed with the Luna stratifica\-tions defined as follows. The fact that the points $a, b\in Y/\!\!/H$ belong to the same Luna stratum means that the normal vector bundles to the unique $H$-orbits closed in the fibers $\pi_{Y/\!\!/H}^{-1}( a)$ and $\pi_{Y/\!\!/H}^{-1}(b)$ orbits are $H$-equivariantly isomorphic. The Luna strata in $Y$ are the sets of the form $\pi_{Y/\!\!/H}^{-1}(L)$, where $L$ is a Luna stratum in $Y/\!\!/H$. The Luna stratifications have the following properties:
\begin{enumerate}[\hskip 9.2mm\rm(i)]
\item the set of all Luna strata is finite;
\item all Luna strata in the varieties $Y/\!\!/H$ and $Y$ are smooth locally closed subvarieties of these varieties;
\item for any Luna stratum $L$ in $Y/\!\!/H$ there exists
an affine variety $F$ endowed with an action of $H$
such that the restriction of the morphism $\pi_{Y/\!\!/H}^{\ }$ to the stratum $\pi_{Y/\!\!/H}^{-1 }(L)$ (called the {\it canonical  morphism} of the Luna stratum $\pi_{Y/\!\!/H}^{-1}(L)$)
is an \'etale trivial bundle
$\pi_{Y/\!\!/H}^{-1}(L)\to L$
with fiber $F$.
\end{enumerate}

In view of (i) and (ii), there are (unique) open Luna strata in $Y/\!\!/H$ and $Y$. They are called the {\it principal strata} and denoted by $(Y/\!\!/H)_{\rm pr}$ and $Y_{\rm pr}$ respectively.

\begin{lemma}\label{lmmm}
We keep the previous notation $H$, $Y$, $\alpha$. Let $y\in Y$ be a point such that
\begin{enumerate}[\hskip 6.2mm \rm(a)]
\item[$({\rm y}_1)$] the orbit $H\!\cdot\!y$ is closed in $Y$;
\item[$({\rm y}_2)$] $H_y=\ker(\alpha)$.
\end{enumerate}
Let $\beta$ be an action of the group $H$ on an algebraic variety $Z$ such that
\begin{enumerate}[\hskip 4.5mm \rm(k)]
\item $\ker(\alpha)=\ker(\beta)$,
\end{enumerate}
and let $\varphi\colon Z\to Y$ be an
$H$-equivariant morphism such that $\varphi^{-1}(H\!\cdot\!y)\neq \varnothing $.
Then for every point $z\in \varphi^{-1}(H\!\cdot\!y)$ the following properties hold:
\begin{enumerate}[\hskip 6.2mm \rm(a)]
\item[$({\rm z}_1)$] the orbit $H\!\cdot\!z$ is closed in $Z$;
\item[$({\rm z}_2)$] $H_ z=\ker(\beta)$.
\end{enumerate}
\end{lemma}

\begin{proof} In view of $({\rm y}_1)$, the nonempty $H$-invariant subset $\varphi^{-1}(H\!\cdot\! y)$ is closed in $Z$ and hence contains the closure $ \overline{H\!\cdot\!z}$ of the orbit ${H\!\cdot\!z}$. Assume that $({\rm z}_1)$ fails, i.e., $\overline{H\!\cdot\!z}\setminus {H\!\cdot\!z}\neq\varnothing$.
Let $v\in \overline{H\!\cdot\!z}\setminus {H\!\cdot\!z}$. Then
\begin{equation}\label{edi}
\dim(H\!\cdot\!v)<\dim(H\!\cdot\!z).
\end{equation}
The restriction of the morphism $\varphi$ to the orbit $H\!\cdot\!v$ is an $H$-equivariant and therefore a surjective morphism $H\!\cdot\!v\to H\!\cdot\! y$. So $\dim(H\!\cdot\!v)\geqslant \dim(H\!\cdot\!y)$, which together with \eqref{edi} gives
\begin{equation}\label{edii}
\dim(H\!\cdot\!y)<\dim(H\!\cdot\!z)
\end{equation}
On the other hand, since $H_z\supseteq \ker(\beta)$, from $({\rm y}_2)$ and
(k) it follows that $\dim(H\!\cdot\!y)\geqslant \dim(H\!\cdot\!z)$. This contradicts \eqref{edii} and proves $({\rm z}_1)$.

The restriction of the morphism $\varphi$ to the orbit $H\!\cdot\!z$ is an $H$-equiva\-riant, and therefore a surjective morphism of $H\!\cdot\!z\to H\!\cdot\!y$. Hence, there exists $h\in H$ for which $\varphi(h\cdot z)=y$. Therefore,
\begin{equation*}
\ker(\alpha)\overset{{\rm (k)}}{=}\ker(\beta)\subseteq H_{h\cdot z}\subseteq H_y\overset{({\rm y}_2)} {=}\ker(\alpha);
\end{equation*}
whence, $H_{h\cdot z}=\ker(\beta)$. Since $H_{h\cdot z}=hH_zh^{-1}$ and $\ker(\beta)$ is normal in $H$, this proves $({\rm z}_2)$.
\end{proof}
\eject

\begin{lemma}\label{ot}
Let a commutative group $H$ act transitively on a set $M$.
\begin{enumerate}[\hskip 4.2mm $({\rm e}_1)$]
\item For every $H$-equivariant mapping $\varphi\colon M\to M$ there is an ele\-ment $h\in H$ such that
\begin{equation}\label{eo}
\mbox{$\varphi(m)=h\cdot m$ for every point $m\in M$.}
\end{equation}
\item For every element $h\in H$, the map $\varphi\colon M\to M$ defined by formula {\rm \eqref{eo}} is $H$-equivariant.
\end{enumerate}
\end{lemma}

\begin{proof}
$({\rm e}_1)$ Fix a point $m_0\in M$. Since the action is tran\-si\-tive, for any point $m\in M$ there is an element $z\in H$ such that $m=z\cdot m_0$. In particular, there is $h\in H$ for which $\varphi(m_0)=h\cdot m_0$.
Since $\varphi$ is $H$-equivariant and $H$ is commutative, we then have $\varphi(m)=\varphi(z\cdot m_0)=z\cdot\varphi(m_0)=z\cdot (h\cdot m_0)=
zh\cdot m_0=hz\cdot m_0=h\cdot(z\cdot m_0)=h\cdot m$.

$({\rm e}_2)$ This follows directly from \eqref{eo} in view of the commutativity of\;$H$.
\end{proof}

\end{section}

\begin{section}{\bf Reduction}\label{reduction}

The proof of Theorem \ref{mate} is based on the following geometric de\-scrip\-tion of the kernel of homomorphism \eqref{action3}:

\begin{lemma}\label{categ}
Let $G$ be a connected affine algebraic group and let $R$ be its closed reductive subgroup.
  The following properties of an element $\sigma\in {\rm Aut}(F_n)$
are equivalent:
\begin{enumerate}[\hskip 4.2mm\rm(a)]
\item
$\sigma$ lies in the kernel of homomorphism \eqref{action3};
\item $\sigma_X^{\ }(\mathcal O)=\mathcal O$ for every closed $R$-orbit $\mathcal O$ in $X$.
\end{enumerate}
\end{lemma}

\begin{proof}
In this case, the variety $X$ is affine, which implies
(see \cite[\S2 and Append.\,1B]{18}) that the morphism $\pi$ is surjective, its fibers are $R$-inva\-riant, and for each point
$b\in X/\!\!/R$, the fiber $\pi^{-1}(b)$ contains a unique closed $R$-orbit $\mathcal O_b$.
It follows from \eqref{pssp} that the restriction of the morphism $\sigma_X$ to the fiber $\pi^{-1}(b)$ is its $R$-equivariant isomorphism with the fiber $\pi^{-1}(\sigma_ {X/\!\!/R}(b))$. In view of the uniqueness of closed orbits in the fibers,
this means that $\sigma_X(\mathcal O_b)=\mathcal O_{\sigma_{X/\!\!/R}(b)}$. Therefore, $\sigma_{X/\!\!/R}(b)=b$ if and only if $\sigma_X(\mathcal O_b)=\mathcal O_b$.
\end{proof}

Under the conditions of Lemma \ref{categ}, the algebra $k[{X}]^R$ of all $R$-invariant elements of the algebra $k[{X}]$ of regular functions on $X$ is finitely generated, ${X }/\!\!/R$ is the affine algebraic variety with the algebra of regular functions $k[{X}/\!\!/R]=k[{X}]^R$, and the comorphism correspond\-ing to morphism \eqref{Luna} is the identity embedding $k[{X}]^R\hookrightarrow k[{X}]$.
This implies that for any reductive closed subgroup $S$ of $G$ containing $R$, the identity embedding $k[X]^S\hookrightarrow k[X]^R$ determines  a dominant morphism
$X/\!\!/R\to X/\!\!/S$. This morphism is ${\rm Aut}(F_n)$-equivariant. Therefore,
the kernel of the action of the group ${\rm Aut}(F_n)$ on $X/\!\!/R$ lies in the kernel of its action
on  $X/\!\!/S$.

In the situation considered in Theorem \ref{mate}, this gives the following. By the assumption, in it, $R$ is a subgroup of some maximal torus $T$ of the group $G$. Therefore, it follows from what has been said that it suffices to prove Theorem \ref{mate} for
\begin{equation}\label{R=T}
R=T.
\end{equation}

In what follows, we assume that the group $G$ satisfies the conditions of Theorem \ref{mate}, i.e., is connected and semisimple. Note that the kernel of the action of the group $T$ on $X$ is $\mathscr{C}(G)$, since $\mathscr{C}(G)\subset T$ (see \cite[13.17,\,Cor. 2(d)]{3}).

\end{section}

\begin{section}{\bf Principal Luna stratum for action of   \boldmath $T$ on $X$}

The variety $X$ is smooth, and the group $T$ is reductive. Therefore, the diagonal action of the torus $T$ on $X$ by conjugation determines the Luna stratifications of the varieties $X$ and $X/\!\!/T$. In what follows, $X_{\rm pr}$ denotes the principal stratum of this stratification of the variety $X$.

\begin{theorem}\label{pLs}
    Let $G$ be a connected semisimple algebraic
group with a maximal torus $T$ acting diagonally by conjugation on the group variety $X$ of the group $G^n$.
\begin{enumerate}[\hskip 4.2mm\rm(a)]
\item The kernel of the specified action is $\mathscr{C}(G)$.
\item The following properties of a point $x\in X$ are equivalent:
\begin{enumerate}[\hskip 0mm\rm $({\rm b}_1)$]
\item $x\in X_{\rm pr}$;
\item the orbit $T\!\cdot\! x$ is closed in $X$, and $T_x=\mathscr{C}(G)$.
\end{enumerate}
\item Each fiber of the canonical morphism of
the Luna stratum $X_{\rm pr}$ is a $T$-orbit equivariantly isomorphic to $T/\mathscr{C}(G)$.
     \item ${\rm codim}_{X}(X\setminus X_{\rm pr})\geqslant n$.
     \end{enumerate}
\end{theorem}
\begin{proof}
Statement (a) is obvious.

(b) Denote by $\mathcal V$ the trivial ${\rm codim}_X(T)$-dimensional vector bundle over $T/\mathscr{C}(G)$. Note that $\dim(T/\mathscr{C}(G))=\dim(T)$ since the group $\mathscr{C}(G)$ is finite (see \cite[14.2.\,Cor.(a)]{3}).

$\cc$ It suffices for us to prove that
\begin{enumerate}[\hskip 4.2mm\rm(i)]
\item the action of the torus $T$ on $X$ under consideration is stable;
\item ${\mathscr C}(G)$ is its stabilizer in general position,
\end{enumerate}
or, in other words, that there is a nonempty open subset of $X$,
for all points $x$ of which property $({\rm b}_2)$ holds.
Indeed, suppose this
subset exists. Due to its openness, its intersection with the
open set $X_{\rm pr}$ is nonempty. Let $x$ be a point of this intersection.
Since $\mathscr{C}(G)$ is the kernel of the action of the group $ T$ on $X$, from  condition $({\rm b}_2)$
it follows that the normal bundle of the orbit $T\cdot x$ is equivariantly isomorphic to
$\mathcal V$. From this and from the definition of the Luna strata it follows that
a closed $T$-orbit from $X$ lies in $X_{\rm pr}$ if and only if its normal bundle is equivariantly isomorphic to $\mathcal V$. In particular, the dimension of this orbit is $\dim(T)$. It remains to note that the $T$-orbit of any point $y\in X_{\rm pr}$ is closed. Indeed, if this were not the case, then the unique  closed $T$-orbit in the fiber $\pi^{-1}_{X/\!\!/T}(\pi_{X/\!\!/T}^ {\ }(y))\subseteq X_{\rm pr}$ lying in its closure had
dimension strictly less than $\dim(T\!\cdot\!y)\leqslant \dim(T)$, which contradicts the $\dim(T)$-dimensionality of this closed orbit.

$\cc$ Let us now prove that properties (i) and (ii) indeed hold. It suffices to prove them for $n=1$. Indeed, suppose that for $n=1$ they are proved, i.e., there is a nonempty open subset of $G$ such that $T$-orbit of every its point $x$ is closed in $G$ and $T_x={\mathscr C}(G)$.
Then, as explained above, $G_{\rm pr}$ is the set of all such points $x$.
Let
$\pi_i\colon X=G^n\to G$ be the natural projection onto the $i$th factor. Applying  Lemma \ref{lmmm} to it, we infer that property $({\rm b}_2)$ holds for each point $x$ of a nonempty set $\pi_i^{-1}(G_ {\rm pr})$, which means that properties (i) and (ii) hold.

 $\cc$ It remains to prove that (i) and (ii) hold for $n=1$. In \cite[6.11]{27},
  it is proved that the action of $G$ on itself by conjugation is stable and its s.g.p. is $T$. From \cite[Thm. and Sect. 3]{Lu1} and the reductivity of $T$, it follows that the natural action of $T$ on $G/T$ is stable. These two facts imply, according to \cite[Prop.\,6]{P89}, that (i) holds for $n=1$.

$\cc$ Let $\Phi$ be the root system of the group $G$ with respect to the torus $T$ in which
subsystems of positive and negative roots with respect to some base in $\Phi$ are fixed.
For any $\alpha\in \Phi$, there is an embedding of algebraic groups
$\varepsilon_\alpha\colon \mathbb G_a\hookrightarrow G$,
such that
\begin{equation}\label{rosu}
t\varepsilon_\alpha(x)t^{-1}=\varepsilon_\alpha(\alpha(t)x)\quad\mbox{for all $t\in T, x\in \mathbb G_a$}
\end{equation}
(see \cite[26.3.\,Thm.]{Hu}, \cite[2.1]{27}). Consider in $G$ the ``big cell'' $\Theta$ (see \cite[28.5\,Prop.]{Hu}), i.e., the set of all elements of the form
\begin{equation}\label{Bcell}
\prod_{\alpha<0}\varepsilon_\alpha(x_\alpha)t \prod_{\alpha>0}\varepsilon_\alpha(x_\alpha),\quad x_\alpha\in \mathbb G_a, t\in T,
\end{equation}
where the factors in the products are taken with respect to some fixed orders on the sets of positive and negative roots.
The set $\Theta$ is open in $G$ and each of its elements can be uniquely written as \eqref{Bcell} (see\,\cite[14.5.\,Prop.(2), 14.14.\,Cor.]{3}, \cite[2.2, 2.3]{27}). In view of \eqref{rosu}, it is $T$-invariant. The set $\Theta^0$ of all elements of the form \eqref{Bcell} with $x_\alpha\neq 0$ for each $\alpha\in\Phi$ has the same properties. Let $a\in\Theta^0$ and $c\in T$. It follows from \eqref{rosu} and the indicated uniqueness that the condition $c\in T_a$ is equivalent to the condition
   \begin{equation}\label{cente}
   \mbox{$c\in \ker(\alpha)$ for all $\alpha\in \Phi$.}
\end{equation}
In turn, it follows from \eqref{rosu}, \eqref{Bcell} and the openness of $\Theta^0$ that \eqref{cente} is equivalent to the property that $c$ belongs to the kernel of the action of $T$ on $G$, i.e., \eqref{cente} is equivalent to the inclusion $c\in \mathscr{C}(G)$. This proves that $T_a=\mathscr{C}(G)$. Hence, (ii) holds for $n=1$. This completes the proof of (b).

(c) This follows from (b), since each fiber of the canonical morphism of any Luna stratum in $X$ contains a unique orbit closed in $X$.

(d) As is explained in the proof of statement (b), the set $X_{\rm pr}$ contains the set $\bigcup_{i=1}^n \pi_i^{-1}(G_{\rm pr})$,
from where we get
\begin{equation}\label{bounda}
X\setminus X_{\rm pr}\subseteq
(G\setminus G_{\rm pr})\times\cdots\times (G\setminus G_{\rm pr})\quad \mbox{($n$ factors).}
\end{equation}
From \eqref{bounda} it follows that $\dim (X\setminus X_{\rm pr})\leqslant n(\dim( G)-1)=\dim(X)-n$. This proves (d).
\end{proof}
\end{section}

\begin{section}{\bf Proofs of Theorems \ref{mate}--\ref{corcor}}\label{proofs}
\label{proofs}

\begin{proof}[Proof of Theorem {\rm \ref{mate}}]
As is explained in Section \ref{reduction}, we can (and shall) assume that equality \eqref{R=T} holds.
Arguing by contradiction, suppose that the kernel of homo\-mor\-phism \eqref{action3} contains an element $\sigma\!\in\!{\rm Aut}(F_n)$, $\sigma\neq e$. The cases $n=1$ and $n\geqslant 2$ will be considered separately: in each of them
the proof is based on the properties that do not hold
in the other.

\vskip 1mm

{\it Case $n=1$.}

The order of ${\rm Aut}(F_1)$ is $2$ and $\sigma(f_1)=f_1^{-1}$, so
\begin{equation}\label{-1}
\sigma_X(g)=g^{-1}\quad \mbox{for each $g\in G=X$.}
\end{equation}
For any element $t\in T$ we have $T\cdot t=t$. In view of  Lemma \ref{categ}, this implies that
$\sigma_X(t)=t$. Together with \eqref{-1} this shows that $t^2=e$ for any $t\in T$. This conclusion contradicts the fact that
the set of orders of elements of the torsion subgroup of any torus of positive dimension is not upper bounded (see\,\cite[8.9.\,Prop.]{3}).

\vskip 1mm

{\it Case $n\geqslant 2$.}

$\cc$ Since the kernel of the considered action
of the torus $T$ on $X$ is ${\mathscr C}(G)$ (see Theorem \ref{pLs}(a)), this action defines a faithful (that is, with trivial kernel) action on $X $ of the torus
\begin{equation}\label{TS}
S:=T/{\mathscr C}(G).
\end{equation}
The orbits of this action of the torus $S$, and hence the categorical quotient and the Luna stratifications are the same as those of the action of the torus $T$. Below, instead of the original action of the torus $T$, we consider the indicated action of the torus $S$.

$\cc$ Theorem \ref{pLs}(b) and Lemmas \ref{categ}, \ref{ot}$({\rm e}_1)$ imply
the existence of a set-theoretic mapping $\psi\colon (X/\!\!/T)_{\rm pr}\to S$ such that
\begin{equation}\label{sss}
\sigma_X^{\ }(x)=\psi(\pi_{X/\!\!/T}^{\ }(x))\cdot x\quad\mbox{for each point $x\in X_{ \rm pr}$.}
\end{equation}
Let us prove that the set-theoretic mapping
\begin{equation}\label{covpsi}
X_{\rm pr}\to S,\quad x\mapsto \psi(\pi_{X/\!\!/T}^{\ }(x))
\end{equation}
is a morphism of algebraic varieties. According to Theorem \ref{pLs}(b) and what was said in part\,(iii) of Section \ref{term}, the canonical morphism $X_{\rm pr}\to (X/\!\!/T)_{ \rm pr}$ is an \'etale trivial bundle with fiber $S$. Since algebraic tori are special groups in the sense of Serre
(see \cite[Prop.\,14]{S58}), this bundle is locally trivial in the Zariski topology.
Hence, $X_{\rm pr}$ is covered by $S$-invariant open sets for which there are $S$-equivariant isomorphisms of them with varieties of the form $U\times S$, where $U$\ is an open subset of $(X/\!\!/T)_{\rm pr}$, and the torus $S$ acts through translations of the second factor. If we identify them by these isomorphisms, then the restriction of mapping \eqref{covpsi} to any of these open sets has the form
\begin{equation*}
\alpha\colon
U\times S\to S, \quad (u, s)\mapsto \psi(u).
\end{equation*}
The issue therefore boils down to proving that
$\alpha$ is a morphism of algebraic varieties. To this end, note that since $\sigma_X^{\ }$ is a morphism, then
\begin{equation*}\label{usus}
U\times S\to U\times S,\quad (u, s)\mapsto (u, \psi(u)s)
\end{equation*}
is also a morphism in view of \eqref{sss}. Hence
\begin{equation*}
\beta
\colon
 U\times S\to S\times S,\quad (u, s)\mapsto (s, \psi(u)s)
\end{equation*}
is a morphism as well.
Moreover,
\begin{equation}
\gamma\colon S\times S\to S,\quad (s_1,s_2)\mapsto s^{-1}_1s_2.
\end{equation}
is a morphism too.
It remains to note that $\alpha=\gamma\circ\beta$.

$\cc$ Thus, there exists a rational mapping $$\theta\colon X \dashrightarrow S,$$ which is defined everywhere on the open set $X_{\rm pr}$ and coincides on it with morphism \eqref{covpsi}. Since $n\geqslant 2$, it follows from Theorem \ref{pLs}(d) that
\begin{equation}\label{cdm2}
{\rm codim}_X(X\setminus X_{\rm pr})\geqslant 2.
\end{equation}

The torus $S$ can be identified with the product of several copies of the group $k^\times$.
Then $\theta$ is given by a set of rational functions $\theta_i\colon X \dashrightarrow k$, which are compositions of the mapping $\theta$ with projections of this product onto the factors. Each $\theta_i$ is regular and does not vanish on $X_{\rm pr}$. Since $X$ is smooth, it follows from this and \eqref{cdm2} that
the divisor of $\theta_i$ on $X$ is zero, that is, $\theta_i$ is regular and does not vanish on the whole of $X$. Thus, we have a morphism $\theta_i: X\to k^\times$. Since $X$ is the group variety of the connected algebraic group $G^n$, it follows from this and from \cite[Thm.\,3]{R61} that $\theta_i$ is the product of a character of this group and a constant. But due to semisimplicity, $G^n$ has no nontrivial characters. Hence $\theta_i$ is a constant. This means that there is an element $s\in S$ for which
$\theta(X)=s$.

$\cc$ Fix an element $t\in T$ that maps to $s$ under the natural surjection $T\to S$ (see \eqref{TS}). We have proven that
$\sigma_X^{\ }(x)=t\cdot x$ for every point $x\in X_{\rm pr}$. Since $X_{\rm pr}$ is open in $X$, this means that
\begin{equation}\label{tst}
\sigma_X^{\ }(x)=t\cdot x\quad \mbox{for every point $x\in X$.}
\end{equation}

Since $\sigma\neq e$, it follows from \cite[Thm.\,2$({\rm b}_1)$]{22} that $\sigma_X\neq {\rm id}_X$. In view of \eqref{tst} and Theorem \ref{pLs}(a), this gives
\begin{equation}\label{notincen}
t\notin\mathscr{C}(G).
\end{equation}

It follows from \eqref{tst}, \eqref{gh}, and \eqref{mor} that for each $i\in\{1,\ldots, n\}$ the following group identity holds
\begin{equation}\label{identity}
\sigma(f_i)(g_1,\ldots, g_n)=tg_it^{-1}\quad\mbox{for any $g_1,\ldots, g_n\in G$}.
\end{equation}
In particular, for each $g\in G$ the equality obtained by substituting $g_1=\ldots=g_n=g$ into \eqref{identity} holds. Since $\sigma(f_i)$ is a noncommutative Laurent monomial in $f_1,\ldots f_n$, this means that there exists an integer $d$ such that the following group identity holds:
\begin{equation}\label{sgs1}
g^d=tgt^{-1}\quad\mbox{for each $g\in G$.}
\end{equation}

Notice that
\begin{equation}\label{neq}
d\neq 1\quad \mbox{and}\quad d\neq -1.
\end{equation}
Indeed, in view of \eqref{sgs1}, if $d=1$, then $t\in{\mathscr C}(G)$ contrary to \eqref{notincen}. If $d=-1$, then
for any $g, h\in G$ the following equality holds:
\begin{equation*}
h^{-1}g^{-1}=(gh)^{-1}\overset{\eqref{sgs1}}{=}t(gh)t^{-1}
=tgt^{-1}tht^{-1}\overset{\eqref{sgs1}}{=}g^{-1}h^{-1},
\end{equation*}
which means that the group $G$ is commutative and contradicts its semisimplicity.

Further, if $r$ is a positive integer, then the following group identity holds:
\begin{equation}\label{idenex}
t^rgt^{-r}=g^{d^r}\quad \mbox{for each $g\in G$.}
\end{equation}
Indeed, \eqref{idenex} becomes \eqref{sgs1} for $r=1$. Arguing by induction, from $t^{r-1}gt^{-r+1}=g^{d^{r-1}}$ we get
\begin{equation*}\label{sgs2}
t^rgt^{-r}{=}t(t^{r-1}gt^{-r+1})t^{-1}=
tg^{d^{r-1}}t^{-1}
\overset{\eqref{sgs1}}{=}(g^{d^{r-1}})^d=g^{d^r},
\end{equation*}
as stated.

Substituting $g=t$ into \eqref{sgs1} and taking into account \eqref{neq}, we conclude that
$t$ is an element of finite order. Let $r$ in \eqref{idenex} be equal to this order. Then \eqref{idenex} turns into the group identity
\begin{equation}\label{idenex1}
e=g^{d^r-1}\quad \mbox{for every $g\in G$.}
\end{equation}

In view of \eqref{neq}, we have $d^r-1\neq 0$. Hence
the group identity \eqref{idenex1} implies that
$G$, and therefore also $T$, is a torsion group the orders of whose elements are upper bounded. We have arrived at the same contradiction as when considering the case $n=1$. This completes the proof of Theorem \ref{mate}.
\end{proof}

\begin{proof}[Proof of Theorem {\rm \ref{mateappl}}]
Theorem \ref{mate} implies (a).
For $n\geqslant 3$, the group ${\rm Aut}(F_n)$ is nonlinear (see \cite{13}) and contains the group $B_n$ (see \cite[Sect. 3.7]{16}). In view of (a), this implies (b) and (c). Since the group $\mathscr {C}(F_n)$ is trivial for $n\geqslant 2$ (see \cite[Chap. I, Prop. 2.19]{15}), the group ${\rm Int}(F_n)$ is isomorphic to $F_n$ and hence is not amenable.
From this it follows\;(d).
\end{proof}

\begin{proof}[Proof of Theorem {\rm \ref{rationa}}]
Let us prove that the group variety of the group $G$ is birationally $T$-equivariantly isomorphic to some
$T$-module. To this end
consider the open set $\Theta$ in $G$ introduced in the proof of Theorem \ref{pLs} and fix the following objects:
\begin{enumerate}[\hskip 4.2mm $\cc$]
\item
one-dimensional $T$-module $L_\alpha$ for every $\alpha\in \Phi$ on which $T$ acts by the formula
\begin{equation}\label{La}
t\cdot \ell=\alpha(t)\ell,\quad t\in T, \ell\in L_\alpha,
\end{equation}

\item nonzero element $\ell_\alpha\in L_\alpha$,

\item a trivial $T$-module $F$ of dimension $\dim(T)$,

\item an open embedding of algebraic varieties
\begin{equation*}\label{embd}
\iota\colon T\hookrightarrow F
\end{equation*}
\end{enumerate}
(it exists because $T$ is a torus).

Consider the $T$-module
\begin{equation*}
V:=\textstyle\bigoplus_{\alpha>0}L_\alpha\oplus F\oplus \bigoplus_{\alpha<0}L_\alpha,
\end{equation*}
where the summands in the direct sums are taken with respect to some fixed orders on
the sets of positive and negative roots. In view of \eqref{rosu}, \eqref{La},
the mapping $\tau\colon \Theta\to V$,
sending each element \eqref{Bcell} to the vector
\begin{equation*}
\textstyle\bigoplus_{\alpha>0} x_\alpha \ell_\alpha\oplus \iota(t)\oplus\bigoplus_{\alpha<0}
x_\alpha\ell_\alpha,
\end{equation*}
is the searched for birational morphism (see \cite[14.4.\,Rem.]{3}).

Consequently,
\begin{equation*}
\tau^n\!:=\!\tau\times\cdots \times \tau\colon \Theta^n\!:=\!\Theta\times\cdots\times\Theta\to V^n\!: =\!V\oplus\cdots\oplus V\quad\mbox{($n$ components)}
\end{equation*}
is also a $T$-equivariant, and therefore, an $R$-equivariant birational mor\-phism.

Since $R$ and $V^n$ are respectively a diagonalizable group and an $R$-module, the field $k(V)^R$ is rational over $k$ (see \cite[Sect 2.9]{PV}). Since $\Theta^n$ is open in $X$, this implies that the field $k(X)^R$ is also rational over $k$. But the action of $T$ on $X$ is stable, and $\mathscr{C}(G)$ is the s.g.p. for it by Theorem \ref{pLs}(b). Since $R$ is reductive, from \cite[Thm. and Sect. 3]{Lu1} it follows that the natural action of $R$ on $T/\mathscr{C}(G)$ is stable. Hence, according to \cite[Prop.\,6]{P89}, the action of $R$ on $X$ is stable.
In view of \cite[Prop. 3.4]{PV}, this implies  that $k(X)^R$ is the field of fractions of the algebra $k[X]^R=k[X/\!\!/R]$. This is what the rationality of the variety
$X/\!\!/R$ means.
\end{proof}

\begin{proof}[Proof of Theorem {\rm \ref{corcor}}]
Let $G={\rm SL}_2$, so that $\dim(G)=3$ and $\dim(T)=1$. It follows from here and from
Theorem \ref{pLs}(c) that $\dim(X/\!\!/T)=3n-1$. Hence, in view of the rationality of the variety $X/\!\!/T$ (Theorem \ref{rationa}), the group ${\rm Aut}(X/\!\!/T)$ embeds into the Cremona group of rank $3n- 1$. The claim of the theorem now follows from Theorem \ref{mateappl} and the fact that every Cremona group embeds
into any Cremona group of a higher rank.
\end{proof}

\end{section}

\end{document}